\documentclass[leqno,a4paper,11pt]{amsart}
\usepackage{amsmath,amsthm,amssymb}
\usepackage[utf8]{inputenc}
\usepackage[T1]{fontenc}

\usepackage{enumerate}
\usepackage{bbm}
\usepackage{color}
\usepackage{mathrsfs}
\usepackage{mathabx}
\usepackage{hyperref}
\usepackage{nicefrac}

\usepackage{url}

  \usepackage[svgnames,x11names]{xcolor} 

  \hypersetup{
    colorlinks, linkcolor={DarkOrchid4},
    citecolor={DarkOrchid4}, urlcolor={Turquoise4}
  }

\makeatletter
\g@addto@macro{\UrlBreaks}{\UrlOrds}
\makeatother

\usepackage[sort,numbers]{natbib}

\providecommand{\noopsort}[1]{} 

\newtheorem{Th}{Theorem}[section]
\newtheorem{Prop}[Th]{Proposition}
\newtheorem{Lemma}[Th]{Lemma}

\theoremstyle{definition}
\newtheorem{Remark}[Th]{Remark}

\newtheorem{Cor}[Th]{Corollary}

\newcommand{\beq}{\begin{equation}}
\newcommand{\eeq}{\end{equation}}

\def\scalar(#1,#2){(#1\mid#2)}

\newcommand{\raz}{\mathbbm{1}}
\newcommand{\jeden}{\mathbbm{1}}

\newcommand{\ov}{\overline}

\newcommand{\R}{{\mathbb{R}}}

\newcommand{\C}{{\mathbb{C}}}

\newcommand{\N}{{\mathbb{N}}}
\newcommand{\NN}{{\mathbb{N}}}

\newcommand{\vep}{\varepsilon}

\newcommand{\mob}{\boldsymbol{\mu}}

\newcommand{\tend}[3][]{\xrightarrow[#2\to#3]{#1}}

\newcommand{\egdef}{:=}

\newcommand{\cav}[2]{d_{#1}\left(#2\right)}
\newcommand{\lav}[2]{d_{#1}^{\text{log}}\left(#2\right)}

\title[M\"obius orthogonality in density]{M\"obius orthogonality in density  for zero entropy dynamical systems}
\author{Alexander Gomilko \and Mariusz Lema\'nczyk \and Thierry de la Rue}
\date{}

\begin{document}

\begin{abstract} It is proved that whenever a zero entropy dynamical system $(X,T)$ has only countably many ergodic measures and $\mob$ stands for the arithmetic M\"obius function then there exists $A=A(X,T)\subset\N$ of logarithmic density one such that for each $f\in C(X)$,
$$
\lim_{A\ni N\to\infty}\frac1N\sum_{n\leq N}f(T^nx)\mob(n)=0$$
uniformly in $x\in X$, in particular, the density version of M\"obius orthogonality holds.\end{abstract}

\maketitle
\normalsize
\thispagestyle{empty}

\section{Introduction} Following P.\ Sarnak \cite{Sa}, we say that a topological system $(X,T)$ is {\em M\"obius orthogonal} if
\beq\label{sarCes}
\lim_{N\to\infty}\frac1N\sum_{n\leq N}f(T^nx)\mob(n)=0
\eeq
for all $f\in C(X)$ and $x\in X$ (here $\mob$ stands for the classical arithmetic M\"obius function).
By the standard trick of summation by parts (which we recall below for the reader's convenience), we obtain that the M\"obius orthogonality of $(X,T)$ implies the {\em logarithmic M\"obius orthogonality} of $(X,T)$:
\beq\label{sarLog}
\lim_{N\to\infty}\frac1{\log N}\sum_{n\leq N}\frac1nf(T^nx)\mob(n)=0
\eeq
for all $f\in C(X)$ and $x\in X$.
The celebrated Sarnak's conjecture \cite{Sa} claims that all zero entropy systems are M\"obius orthogonal, but this statement  has been established only for some selected classes (we refer the reader to the bibliography in survey \cite{Fe-Ku-Le} to see for which classes). In some contrast to this, a considerable progress has been made recently in our understanding of the logarithmic Sarnak's conjecture: Frantzikinakis and Host \cite{Fr-Ho} proved that each zero entropy system whose set of ergodic measures is countable is logarithmic M\"obius orthogonal. Earlier, Tao \cite{Ta2} proved that the logarithmic Sarnak's conjecture is equivalent to the logarithmic version of the classical Chowla conjecture (from 1965) on auto-correlations of the M\"obius function.\footnote{Sarnak's conjecture itself was motivated by the fact that the Chowla conjecture implies Sarnak's conjecture \cite{Ab-Ku-Le-Ru}, \cite{Sa}, \cite{Ta00}.
See also  \cite{Ta1}, \cite{Ta-Te}, \cite{Ta-Te1}, where special cases of the validity of the logarithmic Chowla conjecture have been proved.}

While it looked rather odd to expect that we can say anything interesting about Ces\`aro type averaging knowing the convergence of logarithmic averages, it has been proved in \cite{Go-Kw-Le} that the logarithmic Chowla conjecture implies the validity of the Chowla conjecture along a subsequence. In fact, the result was a consequence of some general mechanism when a certain sequence (in a locally convex space) of logarithmic averages\footnote{The Chowla conjecture can be reformulated using the language of quasi-generic points for invariant measures in a certain shift space; it is then equivalent to the fact that that the empiric measures determined by $\mob$ converge to a certain natural measure which is ergodic, hence to an extremal point, see e.g. the survey \cite{Fe-Ku-Le}. Similarly, we deal with the logarithmic Chowla conjecture.} converges to an extremal point. This approach seems to fail if (what perhaps is natural), we would like to prove that the logarithmic M\"obius orthogonality of a fixed system implies its M\"obius orthogonality along a subsequence.
However, commenting on \cite{Go-Kw-Le},  Tao \cite{Ta} was able to prove a stronger result using a different method (second moment type argument). Namely, he proved that if the logarithmic Chowla conjecture holds then the Chowla conjecture holds along a subsequence of full logarithmic density; in particular of upper density~1.

The aim of this note is to show how to adapt Tao's argument (this is done in Theorem~\ref{banach} below) to be able to apply it to systems satisfying some (seemingly) stronger condition than the logarithmic M\"obius orthogonality and which allows one to deduce M\"obius orthogonality in full logarithmic density. In order to formulate such a result we first recall the {\em strong MOMO} notion introduced  in \cite{Ab-Ku-Le-Ru}. Namely, a system $(X,T)$ satisfies this property if
for all increasing sequences $(b_k)\subset\N$ with $b_{k+1}-b_k\to\infty$, all $(x_k)\subset X$ and $f\in C(X)$, we have
\beq\label{momo00}
\frac1{b_{K+1}}\sum_{k\leq K}\left|\sum_{b_k\leq n<b_{k+1}}f(T^{n-b_k}x_k)\mob(n)\right|\tend{K}{\infty} 0,\eeq
while if under the same assumptions we have
\beq\label{momo01}
\frac1{\log b_{K+1}}\sum_{k\leq K}\left|\sum_{b_k\leq n<b_{k+1}}\frac1nf(T^{n-b_k}x_k)\mob(n)\right|\tend{K}{\infty}  0,\eeq
then we say that $(X,T)$ satisfies the {\em logarithmic strong MOMO} property. It has been proved in \cite{Ab-Ku-Le-Ru} that Sarnak's conjecture is equivalent to the fact that all zero entropy systems enjoy the strong MOMO property.


Note that (\ref{momo00}) is equivalent to
\begin{equation}\label{momo00E}
\lim_{K\to\infty}\,\frac1{b_{K+1}}\sum_{k\leq K}\left\|\sum_{b_k\leq n<b_{k+1}}\mob(n)f\circ T^n\right\|_{C(X)}= 0,
\end{equation}
and (\ref{momo01}) is equivalent to
\begin{equation}\label{momo01E}
\lim_{K\to\infty}\,\frac1{\log b_{K+1}}\sum_{k\leq K}\left\|\sum_{b_k\leq n<b_{k+1}}\frac{\mob(n)}nf\circ T^n\right\|_{C(X)}= 0,
\end{equation}
for all increasing sequences $(b_k)\subset\N$ with $b_{k+1}-b_k\to\infty$ and $f\in C(X)$.
From (\ref{momo00E}) and the triangular inequality we obtain
\begin{equation}\label{unifconv}
\lim_{k\to\infty}\,\frac1{b_{k+1}}\left\|\sum_{n<b_{k+1}}\mob(n)f\circ T^n\right\|_{C(X)}= 0,
\end{equation}
whenever $b_{k+1}-b_k\to \infty$ as $k\to\infty$. Then it is easy to deduce that~\eqref{unifconv} holds for $b_k:=k$, that is, 
the uniform convergence in M\"obius orthogonality~\eqref{sarCes} holds.
Analogously, we obtain that the logarithmic strong MOMO property implies the uniform convergence
in the logarithmic M\"obius orthogonality~(\ref{sarLog}).

Here is our main result:

\begin{Th}\label{t:strongmomo} Assume that a topological system $(X,T)$ satisfies the logarithmic strong MOMO property. Then there exists  $A=A(X,T)\subset\N$ with full logarithmic density such that, for each $f\in C(X)$,
\beq\label{Sar0}
 \lim_{A\ni N\to\infty}\ \left\|\frac1N\sum_{n\leq N}\mob(n)f\circ T^n\right\|_{C(X)}=0.\eeq
In particular, M\"obius orthogonality holds along a subsequence (of $N$) of full logarithmic density.
\end{Th}

One of the main results in \cite{Fr-Ho} states that, if a system $(X,T)$ has zero entropy and if its set of ergodic measures is countable, then
the system is logarithmic M\"obius orthogonal. We will show that such systems satisfy the strong logarithmic MOMO property, hence obtaining the following:\footnote{We were informed by N. Frantzikinakis during the workshop ``Sarnak's Conjecture'' at the American Institute of Mathematics in mid-December 2018 that, independently of us, he can prove Corollary~\ref{c:main0} by modifying some arguments in \cite{Fr-Ho}.}

\begin{Cor}\label{c:main0}
Let $(X,T)$ be a zero entropy ergodic dynamical system such that the set $M^e(X,T)$ of ergodic $T$-invariant measures is countable. Then, there exists $A=A(X,T)\subset\N$ with full logarithmic density along which M\"obius orthogonality holds uniformly in $x\in X$.\end{Cor}

In particular, the above holds for all zero entropy uniquely ergodic systems. Corollary~\ref{c:main0} is slightly surprising even for horocycle flows (in the cocompact case), where we know that M\"obius orthogonality holds \cite{Bo-Sa-Zi} but it is open (see questions in \cite{Fe-Ku-Le}, \cite{Ka-Le-Ul}) whether M\"obius orthogonality holds in its uniform form. By Corollary~\ref{c:main0},  we have that a uniform version holds along a subsequence of logarithmic density~1 (let alone the upper density of this subsequence is~1). A use of~\cite{Fr-Ho1} shows that Corollary~\ref{c:main0} remains valid when $\mob$ is replaced by any multiplicative function which is  strongly aperiodic.

The rest of the note is devoted to give some illustrations how Theorem~\ref{banach} (which is an adaptation of Tao's result) can be applied in other situations. For example, we will show how in the main result in \cite{Go-Kw-Le} we can achieve full logarithmic density. Besides, we note that in the classical Davenport-Erd\"os theorem on the existence of the logarithmic density \cite{Da-Er} of sets of multiples, the upper asymptotic density is achieved along a set of full logarithmic density. Finally, we note in passing the logical implication:  Chowla conjecture of order~2 $\Rightarrow$ PNT along a subsequence of logarithmic density~1.

A few words on basic concepts and notation:
Given a subset $C\subset \N$, we denote by $\delta(C)$ its logarithmic density: $\delta(C):=\lim_{N\to\infty}(1/\log N)\sum_{n\leq N,\, n\in C}\frac1n$,
assuming that the limit exists. 
It is classical that
$
\underline{d}(C)\leq \delta(C)\leq \ov{d}(C)$,
where $\ov{d}(C):=\limsup_{N\to\infty}\frac1N\left|[1,N]\cap C\right|$ stands for the  upper asymptotic density, and similarly the lower (asymptotic) density $\underline{d}(C)$ is defined as the $\liminf$. In fact,
these inequalities are direct consequences of the classical relationship between Ces\`aro averages and harmonic averages: given a sequence $(a_n)$ and setting $s_n:=\sum_{j\leq n}a_j$, $s_0:=0$, we have:
\beq
 \label{sparts}
\begin{split}
\sum_{1\le n\leq N}\frac{a_n}n&=\sum_{1\le n\leq N}\frac1n(s_{n}-s_{n-1})=\\
&\sum_{1\le n\leq N-1}s_n
\left(\frac1n-\frac1{n+1}\right)+\frac{s_N}N=
\sum_{n\leq N-1}\frac{s_n}n\frac1{n+1}+\frac{s_N}N
\end{split}
\eeq
which basically says that the logarithmic averages of $(a_n)$ are the logarithmic averages of Ces\`aro averages (sometimes, we only use the fact that the harmonic averages are convex combinations of Ces\`aro averages).

In what follows when we speak about subsequences of natural numbers, we always mean increasing sequences of natural numbers (so that subsequences are the same as infinite subsets).
In Corollary~\ref{c:main0}, we find a subsequence of full logarithmic density which depends however on the system $(X,T)$. The methods used in this note do not seem to get one universal subsequence along which Sarnak's conjecture (i.e.\ M\"obius orthogonality for zero entropy systems) holds.  We could get such a universal sequence (see Proposition~\ref{p:localS} below) if we were able to prove
Sarnak's conjecture along a full logarithmic density sequence for {\bf each} zero entropy system, that is, by \cite{Ta2}, if the logarithmic Chowla conjecture holds. More precisely:

\begin{Prop}\label{p:localS}
Assume that for each zero entropy dynamical system $(X,T)$ there exists a subsequence $(N_k(X,T))_k$ of natural numbers with $\delta(\{N_k(X,T):\:k\geq1\})=1$ such that
\beq\label{er}
\lim_{k\to\infty}\frac1{N_k(X,T)}\sum_{n\leq N_k(X,T)}f(T^nx)\mob(n)=0\eeq
for all $f\in C(X)$ and $x\in X$. Then there exists a subsequence $(N_k)$ of natural numbers, $\delta(\{N_k:\:k\geq1\})=1$ such that
for each zero entropy dynamical system $(X,T)$,~\eqref{er} holds along $(N_k)$.
\end{Prop}
To see the proof of Proposition~\ref{p:localS}, we have:

a) By assumption and the classical Lemma~\ref{log}, we obtain that for each zero entropy dynamical system $(X,T)$, we have
$\lim_{N\to\infty}(1/\log N)\sum_{n\leq N}\frac1nf(T^nx)\mob(n)=0$
for each $f\in C(X)$ and $x\in X$.

(b) By (a) and Tao's result (``logarithmic Sarnak implies logarithmic Chowla'') \cite{Ta2},  in the space $M(X_{\mob})$ of measures on $X_{\mob}$, we obtain that
$
(1/\log N)\sum_{n\leq N}(1/n)\delta_{S^n\mob}\to \widehat{\nu}_{\mathscr{S}}$,
where we consider the M\"obius subshift $(X_{\mob},S)$ and $\widehat{\nu}_{\mathscr{S}}$ stands for the relatively independent extension of the Mirsky measure $\nu_{\mathscr{S}}$ of the square-free system $(X_{\mob^2},S)$.\footnote{The measure-theoretic investigations of the square-free system $(X_{\mob^2},\nu_{\mathscr{S}},S)$ have been originated by Sarnak \cite{Sa} and Cellarosi and Sinai \cite{Ce-Si}: the Mirsky measure is ergodic and so is its relatively independent extension.}

(c) By (b) and Theorem~\ref{Erg1}, we obtain
that there exists a subsequence $(N_k)$ with
$ \delta(\{N_k:\:k\geq1\})=1$ such that
$(1/N_k)\sum_{n\leq N_k}\delta_{S^n\mob}\to \widehat{\nu}_{\mathscr{S}}.$

(d) By (c) and the proof of the implication ``Chowla implies Sarnak'' in \cite{Ab-Ku-Le-Ru1}, it follows that
for each zero entropy $(X,T)$, we have
$
\lim_{k\to\infty}(1/N_k)\sum_{n\leq N_k}f(T^nx)\mob(n)=0$ for all $f\in C(X)$ and $x\in X$, so Proposition~\ref{p:localS} follows.

\section{Functional formulation of Tao's result} \label{s:smm}

Our aim in this section is to prove a slight extension of Tao's result from \cite{Ta}:

\begin{Th}\label{banach}
Let $(\mathscr{B}_j,\|\cdot\|_j)$, $j=1,2$, be normed vector spaces  and assume that
$\mathscr{B}_1$ is separable.
 Let $(S_k)_{k\ge 1}$ be a
sequence of linear bounded operators from $\mathscr{B}_1$ to $\mathscr{B}_2$, such that for some $M$ we have, for each $f\in\mathscr{B}_1$ and each $k\ge1$,
\begin{equation}\label{bounded}
\|S_kf\|_2\le M\|f\|_1.~\footnote{Assuming that $\mathscr{B}_1$ is Banach and using the Uniform Boundedness Principle, we only need to assume that in~\eqref{bounded} we have $\sup_{k\in\N}\|S_kf\|_2<+\infty$.}
\end{equation}
 Let $\phi$ be a continuous, positive, strictly increasing and convex function on $[0,\infty)$ with $\phi(0)=0$.
 Suppose that there exists a subsequence $(N_s)\subset \N$ such that for any $f\in \mathscr{B}_1$,
setting
\begin{equation}\label{hh}
R_f(H):=\limsup_{s\to\infty}\,\frac{1}{\log {N_s}}
\sum_{1\le n \le N_s} \frac{1}{n}\phi\left( \left\|\frac{1}{H}\sum_{1\le h\le H} S_{n+h}f\right\|_2\right),
\end{equation}
we have
\begin{equation}\label{zero}
\lim_{H\to\infty}\,R_f(H)=0.
\end{equation}
Then there exists a set $\mathcal{N}$ of natural numbers with the property
\begin{equation}\label{property}
\lim_{s\to\infty}\,\frac{1}{\log {N_s}}\,\sum_{\mathcal{N}\ni N\le N_s} \frac{1}{N}=1,
\end{equation}
such that, for any $f\in \mathscr{B}_1$, we have
\begin{equation}\label{lio1A}
\lim_{N\to\infty,\,N\in \mathcal{N}}\,
\left\| \frac{1}{N}
\sum_{1\le n\le N} S_n f \right\|_2=0.
\end{equation}
Moreover, if $(N_s)=\N$, then  $\delta(\mathcal{N})=1$ and
\begin{equation}
\label{lio1AA}
\lim_{N\to\infty}\,
\left\|
\frac{1}{\log N}
\sum_{1\le n\le N} \frac{S_n f}{n} \right\|_2
=0.
\end{equation}
\end{Th}

The proof of the above theorem requires a few lemmas.

\begin{Lemma}\label{S}
Let $G: \N\to \R_+$.
Suppose that for some
$\gamma\in (0,1)$
and for some subsequence $(N_s)\subset\N$, we have
\begin{equation}\label{C1}
\limsup_{s\to \infty}\, \frac{1}{\log {N_s}}\sum_{1\le n \le N_s} \frac{G(n)}{n}
\le \gamma.
\end{equation}
Then, for the set $M\subset \N$ given by
$
M:=\left\{n:\,G(n)< \sqrt{\gamma}\right\}$,
we have
\begin{equation}\label{R11}
\liminf_{s\to\infty}\frac{1}{\log {N_s}}\sum_{n\le N_s,\,n\in M}\,\frac{1}{n}\ge 1-\sqrt{\gamma}.
\end{equation}
\end{Lemma}

\begin{proof}
Let $Q=\N\setminus M$, so that
$
Q=\left\{n:\,G(n)\ge  \sqrt{\gamma}\right\}$.
By Markov's inequality, we obtain
\begin{align*}
\frac{1}{\log {N_s}}\sum_{1\le n \le N_s} \frac{G(n)}{n}
&\ge \frac{\sqrt{\gamma}}{\log {N_s}}\sum_{n\in Q,\, n\le N_s} \frac{1}{n}=\\
&    \frac{\sqrt{\gamma}}{\log {N_s}}\sum_{n\le N_s} \frac{1}{n}   -\frac{\sqrt{\gamma}}{\log {N_s}}\sum_{n\in M,\, n\le N_s} \frac{1}{n},
\end{align*}
hence
\[
\frac{1}{\log {N_s}}\sum_{n\in M,\, n\le N_s} \frac{1}{n}\ge
\frac{1}{\log {N_s}}\sum_{n\le N_s} \frac{1}{n}-\frac{1}{\sqrt{\gamma}\log {N_s}}\sum_{1\le n \le N_s} \frac{G(n)}{n},
\]
and (\ref{R11}) holds in view of~\eqref{C1}.
\end{proof}

\begin{Lemma}\label{Cor1}
Assume that $F: \N\to \R_+$ is bounded and
satisfies
\begin{equation}\label{C10}
\limsup_{s\to \infty}\, \frac{1}{\log {N_s}}\sum_{1\le n \le N_s} \frac{F(n)}{n}
\le \gamma,\quad \gamma\in (0,1),
\end{equation}
for a subsequence $(N_s)$.
Then, for the set $T\subset \N$ given by
\begin{equation}\label{setM0}
T:=\left\{N:\,\frac{1}{N}\sum_{1\le n\le N} F(n)\le \sqrt{\gamma}\right\},
\end{equation}
we have
\begin{equation}\label{R110}
\liminf_{s\to\infty}\frac{1}{\log {N_s}}\sum_{n\in T,n\le N_s}\,\frac{1}{n}\ge 1-\sqrt{\gamma}.
\end{equation}
\end{Lemma}

\begin{proof}
Set $G(n):=\frac{1}{n}\sum_{1\le m\le n} F(m)$.
By (\ref{sparts}), we have
\[
\sum_{1\le n \le N-1}\frac{G(n)}{n+1}
=\sum_{1\le n\le N} \frac{F(n)}{n}-\frac{1}{N}\sum_{1\le n\le N} F(n),
\]
so that, by (\ref{C10}), we  have
$$
\limsup_{s\to\infty}\,\frac{1}{\log {N_s}}\sum_{1\le n \le N_s}\frac{G(n)}{n}=
\limsup_{s\to\infty}\,\frac{1}{\log {N_s}}\sum_{1\le n \le N_s-1}\frac{G(n)}{n+1}
\le \gamma.
$$
Then, by Lemma \ref{S}, we obtain
(\ref{R110}) for the set $T$ defined by (\ref{setM0}).
\end{proof}

\begin{Lemma}[see Tao \cite{Ta}]\label{Sets}
Assume that $M_k\subset \N$, $k\in \N$, and that
there exists an increasing sequence $(N_s)\subset\N$ such that, for each $k$,
\[
\lim_{s\to\infty}\,\frac{1}{\log {N_s}}\sum_{n\in M_k,\,n\le N_s}\,\frac{1}{n}=1,\quad k\in \N.
\]
Then there exists a subset $M\subset \N$ such that
\begin{equation}\label{D1}
\lim_{s\to\infty}\,\frac{1}{\log {N_s}}\sum_{n\in M,\,n\le N_s}\,\frac{1}{n}=1,
\end{equation}
and such that, for any $k\in \N$, there exists ${s_k}$ with
\begin{equation}\label{D2}
M\cap \{n:\,n\ge N_{s_k}\}\subset M_k.
\end{equation}
\end{Lemma}

\begin{proof}
Replacing, if necessary, each $M_k$ by $M_k\cap\bigcap_{k'<k}M_{k'}$, we may assume without loss of generality that $M_{k+1}\subset M_k$, $k\in {\N}$.

Let us choose an increasing sequence $(s_k)$ such that, for each $k$,
\[ s \ge s_k\Rightarrow \frac{1}{\log {N_s}}\sum_{n\in M_{k},\,n\le N_s}\,\frac{1}{n}\ge 1-\frac1{k}.
\]
We set
$
M:=M_1\cap\left[1,N_{s_1}\right] \cup \bigcup_{k=2}^\infty \left(M_k\cap \left(N_{s_{k-1}},N_{s_k}\right]\right)
$
and verify that $M$ satisfies the desired properties (\ref{D1}) and (\ref{D2}).
\end{proof}

{The following is classical.}

\begin{Lemma}\label{log}
Let $(a_n)_{n\ge 1}$ be a bounded sequence 
in a normed vector space $(\mathscr{B},\|\cdot\|)$.
Suppose that there exists a subsequence $\mathcal{N} \subset\N$, $\delta(\mathcal{N})=1$, 
such that
\begin{equation}\label{B}
\lim_{N\to\infty, N\in\mathcal{N}}\,
\left\|
\frac{1}{N}\sum_{1\le n \le N} a_n \right\|
=0.
\end{equation}
Then
\begin{equation}\label{C}
\lim_{N\to\infty}\,
\left\|
\frac{1}{\log N}\sum_{1\le n \le N} \frac{a_n}{n} \right\| =0.
\end{equation}
\end{Lemma}

\begin{proof}
Let $\mathcal{M}=\N \setminus\mathcal{N}$, so that
\begin{equation}
 \label{eq:comp}
\lim_{N\to\infty}\,\frac{1}{\log N}\sum_{n\le N,\,n\in \mathcal{M}} \frac{1}{n}=0.
\end{equation}
By~\eqref{sparts}, we have
\[
\sum_{1\le n\le N} \frac{a_n}{n}=
\mathcal{E}_N+\sum_{1\le n \le N-1}\frac{\mathcal{E}_n}{n+1},\quad\text{where }
\mathcal{E}_n:=\frac{1}{n}\sum_{1\le m\le n} a_m.
\]
Setting $C:=\sup_n \|\mathcal{E}_n\|$, we obtain
\begin{align*}
\left\| \frac{1}{\log N} \sum_{1\le n \le N} \frac{a_n}n \right\|
&\le
\frac{C}{\log N}+
\left\| \frac{1}{\log N} \sum_{n\le N,\,n\in \mathcal{M}} \frac{\mathcal{E}_n}{n+1}\right\|
+ \left\| \frac{1}{\log N} \sum_{n\le N,\,n\in \mathcal{N}} \frac{\mathcal{E}_n}{n+1} \right\|   \\
&  \le\frac{C}{\log N}+
\frac{C}{\log N}\sum_{n\le N,\,n\in \mathcal{M}} \frac{1}{n}
+\left\| \frac{1}{\log N}\sum_{n\le N,\,n\in \mathcal{N}} \frac{\mathcal{E}_n}{n+1}\right\|,
\end{align*}
and then assertion (\ref{C}) follows from (\ref{eq:comp}) and (\ref{B}).
\end{proof}

\begin{proof}[Proof of Theorem~\ref{banach}]
Let $f\in \mathscr{B}_1$.
By (\ref{hh}), (\ref{zero}) and Lemma \ref{Cor1},
we obtain  that for any fixed large $H\in \N$, there exists
a set $\mathcal{N}_{f,H}$ with the property
\[
\liminf_{s\to\infty}\, \frac{1}{\log {N_s}}
\sum_{\mathcal{N}_{f,H}\ni N\le N_s}\frac{1}{N}\ge 1-\sqrt{R_f(H)},
\]
and such that, for $N\in  \mathcal{N}_{f,H}$, we have
\[
\frac{1}{N}\sum_{1\le n\le N}
\phi\left(\left\|\frac{1}{H}\sum_{1\le h\le H} S_{n+h}f\right\|_2\right)
\le
\sqrt{R_f(H)}.
\]
By deleting at most finitely many elements, we may assume that
$\mathcal{N}_{f,H}$ consists
only of elements of size at least $H^2$.
For any $H_0$, if we set
$
\mathcal{N}_{f, \ge H_0}:=\bigcup_{H\ge H_0}\,\mathcal{N}_{f,H},
$
then $\mathcal{N}_{f,\ge H_0}$ satisfies
$
\lim_{s\to\infty}(1/\log {N_s})\sum_{\mathcal{N}_{f,\ge H_0}\ni N\le N_s}\,\frac{1}{N}=1$.
By Lemma~\ref{Sets}, we can
find a set $\mathcal{N}_f$ of natural numbers with
\[
\lim_{s\to\infty}\frac{1}{\log {N_s}}\sum_{\mathcal{N}_f\ni N\le N_s}\,\frac{1}{N}=1,
\]
and such that, for every $H_0$, every sufficiently large element of $\mathcal{N}_f$ lies in $\mathcal{N}_{f,\ge H_0}$.
Thus, for every sufficiently large $N\in \mathcal{N}_f$, one has
\[
\frac{1}{N}\sum_{1\le n\le N}
\phi\left(\left\|\frac{1}{H}\sum_{1\le h\le H} S_{n+h}f\right\|_2\right)
\le\sqrt{R_f(H)},
\]
for some $H\ge H_0$ with $N\ge H^2$. By the monotonicity of $\phi$ and Jensen's inequality, this implies that
\[
\phi\left(\frac{1}{NH}\left\|\sum_{1\le n\le N}
\sum_{1\le h\le H} S_{n+h}f\right\|_2\right)
\le \phi\left(\frac{1}{N}\sum_{1\le n\le N}
\left\|\frac{1}{H}\sum_{1\le h\le H} S_{n+h}f\right\|_2\right)
\]
\[
\le
\frac{1}{N}\sum_{1\le n\le N}
\phi\left(\left\|\frac{1}{H}\sum_{1\le h\le H} S_{n+h}f \right\|_2
\right)\le \sqrt{R_f(H)},
\]
so that, setting $\psi(s):=\phi^{-1}(\sqrt{s})$, $s>0$, we get
\begin{equation}\label{LLL}
\frac{1}{NH}\left\|\sum_{1\le n\le N}
\sum_{1\le h\le H} S_{n+h}f\right\|_2\le \psi(R_f(H))\to 0,\quad H\to\infty.
\end{equation}

Next, by computation, we get
\[
\sum_{1\le n\le N}
\sum_{1\le h\le H} S_{n+h}f
= \sum_{1\le h\le H}
\sum_{1\le n\le N} S_{n+h}f
=\sum_{1\le h\le H}\sum_{m=h+1}^{h+N} S_m f=
\]
\[
H\sum_{m=1}^N  S_mf
+\sum_{1\le h\le H}\sum_{m=N+1}^{N+h}  S_mf
-\sum_{1\le h\le H}\sum_{1\le m\le h}  S_mf,
\]
which gives
\begin{multline*}
\frac{1}{N}\sum_{m=1}^N  S_mf
=\frac{1}{HN}\sum_{1\le n\le N}
\sum_{1\le h\le H} S_{n+h}\\
-\frac{1}{HN}\sum_{1\le h\le H}\sum_{m=N+1}^{N+h}  S_mf
+\frac{1}{HN}\sum_{1\le h\le H}\sum_{1\le m\le h} S_mf.
\end{multline*}
Now, let us fix $H_0$. For any sufficiently large $N\in \mathcal{N}_f$, there exists $H\ge H_0$, with $H^2\le N$, such that $N\in\mathcal{N}_{f,H}$.
The triangular inequality yields
\begin{multline*}
\frac{1}{N}\left\|\sum_{m=1}^N  S_mf\right\|_2\le
\frac{1}{HN}\left\|\sum_{1\le n\le N}
\sum_{1\le h\le H} S_{n+h}f
\right\|_2\\
+\frac{1}{HN}\left\|\sum_{1\le h\le H}\sum_{m=N+1}^{N+h}  S_mf\right\|_2
+\frac{1}{HN}\left\|\sum_{1\le h\le H}\sum_{1\le m\le h} S_mf\right\|_2.
\end{multline*}
Using (\ref{LLL}), we can upper bound the first term in the right-hand side by $\psi(R_f(H))$. The sum of the other two terms can be upper bounded by
\[ \frac{M\|f\|_1}{HN}
\left(\sum_{1\le h\le H}\sum_{m=N+1}^{N+h} 1
+\sum_{1\le h\le H}\sum_{1\le m\le h} 1\right) \le \frac{2HM\|f\|_1}{N} \le
\frac{2M\|f\|_1}{H}. \]
We thus get that
\[
 \frac{1}{N}\left\|\sum_{m=1}^N  S_mf\right\|_2\le \sup_{H\ge H_0} \psi(R_f(H))+\frac{2M\|f\|_1}{H_0}.
\]
By letting $H_0$ go to infinity, we conclude that
$
\lim_{N\to\infty,\,N\in \mathcal{N}_f}\,\left\|(1/N)\sum_{m=1}^N S_m f\right\|_2=0$.

Let now $(f_k)_{k\ge 1}$ be a dense set in $\mathscr{B}_1$.
Then, by  Lemma~\ref{Sets}, we get a set $\mathcal{N}$ of natural
numbers with the property
$\lim_{s\to\infty}(1/\log N_s)\sum_{\mathcal{N}\ni N\le N_s}\,\frac{1}{N}=1$,
such that for any $k\in \N$,
\[
\lim_{N\to\infty,\,N\in \mathcal{N}}\frac{1}{N}\left\|\sum_{m=1}^N  S_mf_k\right\|_2=0.
\]
Take
$g\in \mathscr{B}_1$.
Then for any $\epsilon>0$ there exists $f_k$ such that
$
\|g-f_k\|_1\le \epsilon/M,
$
and then
\[
\frac{1}{N}\left\|\sum_{m=1}^N  S_m g\right\|_2\le
\frac{1}{N}\left\|\sum_{m=1}^N  S_mf_k\right\|_2+
\epsilon,
\]
which, by the above, yields
\[
\lim_{N\to\infty,\,N\in \mathcal{N}}\frac{1}{N}\left\|\sum_{m=1}^N  S_mg\right\|_2=0.
\]
So, statement (\ref{lio1A})
is proved. Assertion
(\ref{lio1AA}) follows from (\ref{lio1A}) by
Lemma \ref{log}.
\end{proof}

Note that in the sequel we will use
Theorem~\ref{banach} mainly for $\phi(s)=s$ (or sometimes for $\phi(s)=s^2$).

\section{Proof of Theorem~\ref{t:strongmomo}}

 To prove Theorem  \ref{t:strongmomo}  we will show a stronger result:

\begin{Th}\label{t:strongmomoA}
Let $(\mathscr{B}_j,\|\cdot\|_j)$, $j=1,2,$ be normed vector spaces and assume that $\mathscr{B}_1$ is separable.
 Let $(S_k)_{k\ge 1}$ be a
sequence of linear bounded operators from $\mathscr{B}_1$ to $\mathscr{B}_2$, such that for some $M>0$ we have, for each $f\in\mathscr{B}_1$ and each $k\ge1$,
\[
\|S_kf\|_2\le M\|f\|_1.
\]
Assume that $(S_k)_{k\ge 1}$  satisfies the property:
for all increasing sequences $(b_k)\subset\N$ with $b_{k+1}-b_k\to\infty$ and $f\in \mathscr{B}_1$, we have
\[
\lim_{K\to\infty}\,\frac1{\log b_{K+1}}\sum_{k\leq K}\left\|\sum_{b_k\leq n<b_{k+1}}\frac1n S_nf\right\|_2=0,
\]
 Then there exists  $A\subset\N$ with full logarithmic density: $\delta(A)=1$, such that for each $f\in \mathscr{B}_1$,
\[
 \lim_{A\ni N\to\infty}\ \left\|\frac1N\sum_{n\leq N} S_nf\right\|_2=0.
\]
\end{Th}

The proof of Theorem~\ref{t:strongmomoA} will use the following intermediate result.

\begin{Prop}\label{l:prawie}
Let $(\mathscr{B}_j, \|\cdot\|_j)$, $j=1,2$ be normed vector spaces.
Let $(S_n)_{n\ge 1}$ be a sequence of linear bounded operators from $\mathscr{B}_1$ to $\mathscr{B}_2$,
such that
for any $(b_k)\subset \N$ with $0<b_{k+1}-b_k\to\infty$ as $k\to\infty$,
and any $f\in \mathscr{B}_1$ 
we have
\begin{equation}
\label{eq:premisse}
 \frac{1}{\log b_{K+1}}\sum_{k\le K}
\left\|\sum_{b_k\le n< b_{k+1}}\frac{1}{n}S_nf\right\|_2\tend{K}{\infty} 0.
\end{equation}
Then
\begin{equation}\label{MR}
\ov\lim_{H\to\infty}\ov\lim_{N\to\infty}\,
\frac{1}{\log N}\sum_{n\le N}\frac{1}{n}
\left\|\frac{1}{H}\sum_{h\le H} S_{n+h}f\right\|_2=0.
\end{equation}
\end{Prop}

We will prove this result by contraposition, and we will need the following lemma.

\begin{Lemma}[Diagonalization Lemma]\label{Ld}
Consider a family of sequences  $(g_{n,m})\subset\R_{+}$, $m> n,$ $m,n\ge 1$.
Suppose that for some families $(b_{k,\ell})_{k,\ell\ge 1}\subset \N$
with
\begin{equation}\label{bbb}
0<b_{k,\ell}<b_{k+1,\ell},\quad
\overline{\lim}_{\ell\to\infty}\underline{\lim}_{k\to\infty}\, (b_{k+1,\ell}-b_{k,\ell})=\infty,
\end{equation}
and  some $\gamma>0$, we have
\begin{equation}\label{delta}
\overline{\lim}_{K\to\infty}\,\frac{1}{\log b_{K+1,\ell}}\sum_{k=1}^K
g_{b_{k,\ell},b_{k+1,\ell}}\ge \gamma,\;  \text{ for all }\ell\in \N.
\end{equation}
Then there exists a sequence $(b_k)_{k\ge 1}\subset\N$ such that
\begin{equation}\label{AL}
0<b_k<b_{k+1},\quad b_{k+1}-b_k\tend{k}{\infty}\infty,
\end{equation}
and
\begin{equation}\label{BL}
\overline{\lim}_{K\to\infty}\,
\frac{1}{\log b_{K+1}}
\sum_{k=1}^{K}
g_{b_{k}, b_{k+1}}\ge \gamma/2.
\end{equation}
\end{Lemma}

\begin{proof}
Note that by (\ref{bbb}), we have $b_{k+1,1}-b_{k,1}\ge 1$ for each $k\ge1$, and that without loss of generality we may assume that for each $\ell\ge1$, we have
$\underline{\lim}_{k\to\infty}\, (b_{k+1,\ell}-b_{k,\ell})\ge \ell$.
By (\ref{delta}) for $\ell=1$, we can choose the value $K_1$ so that
\[
\frac{1}{\log b_{K_1+1,1}}
\sum_{k=1}^{K_1}
g_{b_{k,1},b_{k+1,1}}\ge \gamma/2
\]
and take then, for $k=1,\dots,K_1+1$,
$b_k:=b_{k,1}$,
to obtain
\[
\frac{1}{\log b_{K_1+1}}
\sum_{k=1}^{K_1}
g_{b_{k},b_{k+1}}\ge \gamma/2.
\]

We continue the above process by induction. Suppose that for some $\ell\ge 1$ we already have sequences $0=K_0<K_1<\cdots<K_\ell$, $1\le b_1<b_2<\cdots<b_{K_{\ell}+1}$,
satisfying, for each $s=1,\ldots,\ell$,
\begin{equation}\label{s1}
b_{k+1}-b_k\ge s,\quad k=K_{s-1}+1,\dots, K_s,
\end{equation}
and
\begin{equation}\label{s2}
\frac{1}{\log b_{K_s+1}}
\sum_{k=1}^{K_s}
g_{b_k, b_{k+1}}\ge \gamma/2.
\end{equation}
Then, at step $\ell+1$, we first choose  $N_{\ell+1}$ large enough so that
$
b_{K_\ell+2+N_{\ell+1},\ell+1}\ge b_{K_\ell +1}+\ell+1$,
and, for each $k\ge K_\ell +2+N_{\ell+1}$,
$
b_{k+1,\ell+1}-b_{k,\ell+1}\ge \ell+1$.
Then, by (\ref{delta}), we can choose $K_{\ell+1}>K_\ell +2$ large enough so that
\[
\frac{1}{\log b_{K_{\ell+1}+1+N_{\ell+1},\ell+1}}
\sum_{k=K_\ell +2+N_{\ell+1}}^{K_{\ell+1}+N_{\ell+1}}
g_{b_{k,\ell+1},b_{k+1,\ell+1}}\ge \gamma/2,
\]
and we take for $k=K_\ell +2,\dots,K_{\ell+1}+1:$
\[
b_k:=b_{k+N_{\ell+1},\ell+1}.
\]
Then assertions (\ref{s1}) and (\ref{s2}) are valid up to $s=\ell+1$,
and this allows us to construct inductively the sequences $(b_k)_{k\ge 1}$ with the required properties
(\ref{AL}), (\ref{BL}).
\end{proof}

\begin{proof}[Proof of Proposition~\ref{l:prawie}]
Suppose that for some $f\in \mathscr{B}_1$, (\ref{MR}) does not hold. Then, there exist $\gamma>0$
and a sequence $(H_{\ell})$, $H_{\ell}\to\infty$ as $\ell\to\infty$, such that for each $\ell\geq1$, we have
\[
\overline{\lim}_{N\to\infty}\,
\frac{1}{\log N}\sum_{1\le n\le N}\frac{1}{n}
\left\|\frac{1}{H_{\ell}}\sum_{1\le h\le H_{\ell}} S_{n+h}f \right\|_2\ge \gamma.
\]
We write this in the form
\[
\overline{\lim}_{N\to\infty}\,
\frac{1}{\log N}\frac{1}{H_{\ell}}
\sum_{r=0}^{H_{\ell}-1} \sum_{n\le N,\,n\equiv r\;[H_{\ell}]}\frac{1}{n}
\left\|\sum_{1\le h\le H_{\ell}} S_{n+h}f \right\|_2\ge \gamma.
\]
Then there exists $r_{\ell}$, $0\le r_{\ell}<H_{\ell}$, such that
\[
\overline{\lim}_{N\to\infty}\,
\frac{1}{\log N}
\sum_{n\le N,\,n\equiv r_{\ell}\;[H_{\ell}]}\frac{1}{n}
\left\|\sum_{1\le h\le H_{\ell}} S_{n+h}f \right\|_2\ge \gamma.
\]

Using
\[
\frac{1}{n}=\frac{1}{n+h}+\frac{h}{n(n+h)},\quad
\lim_{N\to\infty}\,\frac{1}{\log N}\sum_{n\le N,\,n\equiv {r_\ell}\;[H_{\ell}]}\frac{{H^2_{\ell}}}{n^2}=0,
\]
we obtain
\[
\overline{\lim}_{N\to\infty}\,
\frac{1}{\log N}
\sum_{n\le N,\,n\equiv r_{\ell}\;[H_{\ell}]}
\left\|\sum_{1\le h\le H_{\ell}} \frac{1}{n+h}S_{n+h}f\right\|_2\ge \gamma,
\]
or
\[
\overline{\lim}_{N\to\infty}\,
\frac{1}{\log N}
\sum_{n\le N,\,n\equiv r_{\ell}\;[H_{\ell}]}
\left\|\sum_{n < m\le n+H_{\ell}} \frac{1}{m}S_m f\right\|_2\ge \gamma.
\]
Rewrite this inequality in the form
\[
\overline{\lim}_{N\to\infty}\,
\frac{1}{\log N}
\sum_{k=1}^{[N/H_{\ell}]}
\left\|\sum_{kH_{\ell}+r_{\ell}< m\le (k+1)H_{\ell}+r_{\ell}} \frac{1}{m} S_mf\right\|_2\ge \gamma,
\]
and take, for a fixed $\ell$, the sequence $(b_{k,\ell})_{k\ge 1}$ defined by $b_{k,\ell}:=kH_{\ell}+r_{\ell}+1$.
Setting $K_N:=[N/H_{\ell}]$ for each $N$, we have
$
\log {b_{K_N+1,\ell}}/\log N \tend{N}{\infty} 1$,
hence
\[
\limsup_{K\to\infty}\,
\frac{1}{\log {b_{K+1,\ell}}}
\sum_{k\le K}
\left\|\sum_{b_{k,\ell}\le m < b_{k+1,\ell}} \frac{1}{m} S_mf \right\|_2\ge \gamma.
\]

We can now apply  the Diagonalization Lemma~\ref{Ld} with the sequences
\[
 g_{n,m}:=\left\|   \sum_{n\le j < m} \frac1{j} S_j f   \right\|_2.
\]
We obtain that there exists a sequence $(b_k)_{k\ge 1}$ with
$
0<b_{k+1}-b_k\tend{k}{\infty}\infty,
$
such that
\[
\overline{\lim}_{K\to\infty}\,
\frac{1}{\log {b_{K+1}}}
\sum_{k\le K}
\left\|\sum_{b_{k}\le m< b_{k+1}} \frac{1}{m} S_m f \right\|_2\ge \gamma/2.
\]
Hence~\eqref{eq:premisse} is not satisfied
\end{proof}


\begin{proof}[Proof of Theorem~\ref{t:strongmomoA}]
By Proposition~\ref{l:prawie}, we obtain that if we set
$$
R_f(H):=\limsup_{N\to\infty} \frac1{\log N}\sum_{1\le n\leq N}\frac1n\left\|\frac1H\sum_{1\le h\leq H}S_{n+h}f\right\|_{2},
$$ then
$\lim_{H\to\infty}R_f(H)=0$, so the result follows from  Theorem~\ref{banach} with $\phi(s)=s$.
\end{proof}

\begin{proof}[Proof of Theorem~\ref{t:strongmomo}] We apply Theorem~\ref{t:strongmomoA}
 to  $\mathscr{B}_1=\mathscr{B}_2=C(X)$ and $S_k(f):=\mob(k) f\circ T^k$.
\end{proof}

\section{Proof of Corollary~\ref{c:main0} and related results}

Recall that a point $y$ in a topological dynamical system $(Y,S)$ is \emph{quasi-generic} for some measure $\nu$ if, for some subsequence $(N_k)$ of integers and all $f\in C(Y)$, we have
\[
 \frac{1}{N_k} \sum_{1\le n\le N_k} f(S^ny) \tend{k}{\infty} \int_Y f\,d\nu.
\]
Likewise, we say that $y$ is \emph{logarithmically quasi-generic for $\nu$}  if, for some subsequence $(N_k)$ of integers and for all $f\in C(Y)$, we have
\[
 \frac{1}{\log N_k} \sum_{1\le n\le N_k} \frac{1}{n} f(S^ny) \tend{k}{\infty} \int_Y f\,d\nu.
\]
Observe that any measure for which $y$ is logarithmically quasi-generic is $S$-invariant.

We will use here the following result from~\cite{Fr-Ho} (see the remark after Theorem~1.3 therein).

\begin{Th}[Frantzikinakis and Host]
\label{thm:fh}
Let $(Y,S)$ be a topological dynamical system, and let $y\in Y$. Assume that, for any measure $\nu$ for which $y$ is logarithmically quasi-generic, the system $(Y,\nu,S)$ has zero entropy and countably many ergodic components. Then for any $g\in C(Y)$, we have
\begin{equation}
 \label{eq:fhy}
 \lim_{N\to\infty} \frac{1}{\log N} \sum_{1\le n\le N} \frac{1}{n} g(S^ny) \mob(n) = 0.
\end{equation}
\end{Th}

\begin{proof}[Proof of Corollary~\ref{c:main0}]

Let us consider a dynamical system $(X,T)$ with zero topological entropy, and such that $M^e(X,T)$ is countable. In view of Theorem~\ref{t:strongmomo}, all we need to prove is that $(X,T)$ satisfies the logarithmic strong MOMO property. That is, we fix an increasing sequence $(b_k)\subset\N$ with $b_{k+1}-b_k\to\infty$ (and we assume without loss of generality that $b_1=1$), a sequence $(x_k)\subset X$ and $f\in C(X)$,
and we have to show the following convergence
\beq\label{momo1}
\frac1{\log {b_{K+1}}}\sum_{k\leq K} \left|
\sum_{b_k\leq n<b_{k+1}}\frac1n f(T^nx_k)\mob(n)\right|\tend{K}{\infty} 0.\eeq
According to \cite[Lemma 18]{Ab-Ku-Le-Ru}, it is sufficient to show that
\beq\label{momo2}
\frac1{\log {b_{K+1}}}\sum_{k\leq K} e_k
\sum_{b_k\leq n<b_{k+1}}\frac1n f(T^nx_k)\mob(n)\tend{K}{\infty} 0,
\eeq
where
$e_k\in \Sigma_3:=\{e^{2\pi i j/3}:\:j=0,1,2\}$
is chosen so that the product
\[ e_k
\sum_{b_k\leq n<b_{k+1}}\frac1n f(T^nx_k)\mob(n) \]
belongs to the closed cone $\{0\}\cup\{z\in\C:\:{\rm arg}(z)\in [-\pi/3,\pi/3]\}$.

In order to show~\eqref{momo2}, we consider the space $Y:=(X\times \Sigma_3)^{\N}$ with the shift $S$, and in this system the point
$y=(y_n)_{n\in\N}$ defined by
\[
 y_n := (T^n x_k, e_k) \text{ if } b_k\le n<b_{k+1} \quad (k\ge1).
\]
Let $\nu$ be a measure for which $y$ is logarithmically quasi-generic.
The same argument as in~\cite{Ab-Ku-Le-Ru} (see the proof of  (P2 $\Rightarrow$ P3)) shows that $\nu$ must be concentrated on the set of sequences of the form
\[ \bigl( (x,a), (Tx,a), (T^2x,a),\ldots \bigr) \quad (x\in X, a\in \Sigma_3).\]
Now, let us consider an ergodic component $\rho$ of $\nu$. The marginal of $\rho$ on the first coordinate $x$ must be an ergodic $T$-invariant measure on $X$. By assumption, there are only countably many of them, and all of them give rise to  zero-entropy systems. The marginal of $\rho$ on the second coordinate $a$ is one of the three Dirac measures $\delta_{1}$, $\delta_{e^{i2\pi/3}}$, $\delta_{e^{i4\pi/3}}$. Moreover these two marginals must be independent by the disjointness of ergodic systems with the identity, thus, these two marginals completely determine $\rho$. Hence, we see that there can be only countably many possible ergodic components of $\nu$, and all of them have zero entropy. Thus $y$ satisfies the assumptions of Theorem~\ref{thm:fh}, and we have~\eqref{eq:fhy} for each $g\in C(Y)$. In particular, if we take the continuous function $g$ defined by $g(z):= a_0f(z_0)$ for each $z =((z_0,a_0),(z_1,a_1),\ldots)\in Y$, we obtain~\eqref{momo2}.
\end{proof}

\begin{Remark}\label{r:commonmomo}  We can characterize uniform convergence for M\"obius orthogonality in terms of a MOMO type convergence. Indeed:

\noindent {\em The uniform convergence in M\"obius orthogonality~\eqref{sarCes} holds if and only if for all $(b_k)$ satisfying  $b_k/b_{k+1}\to 0$,
we have
\[
\frac1{b_{k+1}}\left\|\sum_{b_k\leq n<b_{k+1}} \mob(n)f\circ T^n\right\|_{C(X)}\tend{k}{\infty} 0.
\]}

To see this equivalence, it is sufficient to note that for each $x\in X$,

\begin{multline*}
\frac1{b_{k+1}} \sum_{b_k\leq n< b_{k+1}}f(T^nx)\mob(n)=
\frac1{b_{k+1}} \sum_{1\leq n< b_{k+1}}f(T^nx)\mob(n)\\
- \frac{b_k}{b_{k+1}}\frac1{b_k} \sum_{1\leq n< b_k}f(T^nx)\mob(n).
\end{multline*}

%
%



Remark also that the same arguments work for the logarithmic averages (replacing $b_k/b_{k+1}\to 0$ with $\log b_k/\log b_{k+1}\to0$).
\end{Remark}

\section{Miscellanea}
\subsection{Ergodic measures} \label{s:ergmes}
In this section, we show that Tao's approach persists, if we consider the main observation from \cite{Go-Kw-Le}.

Let $(X,T)$ be a dynamical system. Given $x\in X$ and $n\in\mathbb{N}$, we write $\delta_{T^n(x)}$ for the Dirac measure concentrated at the point $T^n(x)$.
Let
\[
\mathcal{E}(x,N):=\frac{1}{N}\sum_{1\le n\le N} \delta_{T^{n}(x)}, \quad
\mathcal{E}^{\rm log}(x,N):=\frac1{\log N}\sum_{1\le n\le N}\frac1n \delta_{T^{n}(x)}.
\]
We consider here the convergence of these empirical measures in the weak* topology. Note that any accumulation point of the above sequences is always a $T$-invariant probability measure on $(X,T)$.

\begin{Th}\label{Erg1}
Suppose that for some $x\in X$ and some subsequence $(N_s)_{s\ge 1}$ of natural numbers, we have
\begin{equation}\label{Co1}
\lim_{s\to\infty}\,\mathcal{E}^{\rm log}(x,N_s)=\kappa,\quad \text{where $\kappa$ is \emph{ergodic}.}
\end{equation}
Then there exists a set $\mathcal{N}$ of natural
numbers with the property
\begin{equation}\label{Prop11}
\lim_{s\to\infty}\,\frac{1}{\log {N_s}}\sum_{{\mathcal N}\ni N\le N_s}\,\frac{1}{N}=1,
\end{equation}
such that
\begin{equation}\label{MMM}
\lim_{N\to\infty,\,N\in \mathcal{N}}\,\mathcal{E}(x,N)=\kappa.
\end{equation}
\end{Th}

\begin{proof}
The condition
(\ref{Co1}) means that for any $f\in C(X)$ we have
\begin{equation}\label{CC2}
\lim_{s\to\infty}\,\frac1{\log {N_s}}\sum_{n=1 }^{N_s}\frac {1}{n} f(T^n u)=\int_X f\,d\kappa.
\end{equation}

Let $f\in C(X)$ be fixed, and set $Sf:=\int_X f\,d\kappa$.
For $H\in \N$, we consider the limiting of the second moment
\[
R_f(H):=\lim_{s\to\infty}\, \frac{1}{\log {N_s}}\sum_{1\le n \le N_s} \frac{1}{n}\left|
\frac{1}{H}\sum_{m=1}^H f(T^{n+m}x)-Sf\right|^2.
\]
The limit does exist by~\eqref{Co1}, as the internal is given by a continuous function sum sampled at $x$.
So, by condition (\ref{CC2}), we have
\begin{equation}\label{RR1}
R_f(H)=\int_X
\left|\frac{1}{H}\sum_{m=1}^H f(T^{m} x)-Sf\right|^2\,d\kappa(x).
\end{equation}
Hence, directly by the von Neumann ergodic theorem,
and using the ergodicity of $\kappa$, we obtain
$
\lim_{H\to\infty}\,R_f(H)=0$.
Take now $\mathscr{B}_1=C(X)$, $\mathscr{B}_2=\C$ with the sequence of functionals
$
S_kf:= f(T^{k}x)-Sf$, $k\in\N$,
and we obtain statement (\ref{MMM}) by Theorem \ref{banach}.
\end{proof}

\subsection{Davenport-Erd\"os theorem}
Davenport-Erd\"os theorem \cite{Da-Er} is the fact that, given $\mathscr{B}\subset\N$, the $\mathscr{B}$-free set $\mathcal{F}_{\mathscr{B}}$, i.e.\ the set of those numbers that have no divisor in $\mathscr{B}$,  has logarithmic density and, moreover, $\delta(\mathcal{F}_{\mathscr{B}})=\ov{d}(\mathcal{F}_{\mathscr{B}})$. In fact, see \cite{Dy-Ka-Ku-Le}, the point $\jeden_{\mathcal{F}_{\mathscr{B}}}$ is {\bf logarithmically generic} for the relevant Mirsky measure which is ergodic.
Hence, by Theorem~\ref{Erg1}, we obtain that the upper asymptotic density
$
\ov{d}(\mathcal{F}_{\mathscr{B}})$
is obtained along a subsequence of logarithmic density~1.
We can however obtain this result in an elementary way.
Indeed, for a subset $A\subset\N$ and $N\in\N$, set
$\cav{N}{A} \egdef \dfrac{1}{N}\sum_{1\le n\le N} \raz_A(n)$ and
$\lav{N}{A} \egdef \dfrac{1}{\log N}\sum_{1\le n\le N} \frac1n\raz_A(n)$.
More generally, given
 $a=(a_n)_{n\in\NN}$ a sequence of real numbers, set
 $
   \cav{N}{a} \egdef \dfrac{1}{N}\sum_{1\le n\le N} a_n$ and
   $\lav{N}{a} \egdef \dfrac{1}{\log N}\sum_{1\le n\le N} \frac{a_n}{n}$.

 \begin{Prop}
  Let $a=(a_n)_{n\in\NN}$ be a bounded sequence of real numbers, and let $\ell\egdef\limsup_{N\to\infty}\cav{N}{a}$.
  Assume that $\lim_{N\to\infty}\lav{N}{a}=\ell$. Then there exists $B\subset\NN$ with $\lim_{N\to\infty} \lav{N}{B}=1$ such that
   $\lim_{B\ni N\to\infty} \cav{N}{a}=\ell$.
 \end{Prop}

 \begin{proof}

  \textbf{Step 1} For any $\varepsilon>0$, set $B_\varepsilon\egdef\{N\in\NN: \cav{N}{a} > \ell-\varepsilon\}$. Then $\lim_{N\to\infty}\lav{N}{B_\varepsilon}=1$.
  Indeed, let us introduce the sequence $b=(b_n)_{n\in\NN}$ defined by $b_1\egdef 0$, and for each $n\ge2$, $b_n\egdef \cav{n-1}{a}$.
  The Abel summation formula yields
  \[
   \lav{N}{a} = \dfrac{\cav{N}{a}}{\log N} + \lav{N}{b}.
  \]
  By assumption, we thus have $\lim_{N\to\infty} \lav{N}{b}=\ell$.
  \newcommand{\wb}{\widetilde{B_\varepsilon}}
  \newcommand{\wa}{\widetilde{A_\varepsilon}}
  Let
  \[
   \wb \egdef \{N\in\NN: b_N>\ell-\varepsilon\} (= B_\varepsilon +1)\text{
  and }
   \wa\egdef\NN\setminus \wb =  \{N\in\NN: b_N \le \ell-\varepsilon\}.
  \]
  In the computation of $\lav{N}{b}$, the contribution of $\wa$ is bounded above by
  \[ (\ell-\varepsilon) \lav{N}{\wa}=(\ell-\varepsilon) \left(1-\lav{N}{\wb}\right).\]
  On the other hand, using the fact that $\limsup_{N\to\infty} b_N = \ell$, the contribution of
  $\wb$ to $\lav{N}{b}$ is bounded above by $\ell \lav{N}{\wb} + o(1)$.
 Therefore,  we have for each $N\in\NN$
  \begin{align*}
   \lav{N}{b} &\le (\ell-\varepsilon) \left(1-\lav{N}{\wb}\right) + \ell \lav{N}{\wb} + o(1)\\
    &= \varepsilon \lav{N}{\wb} + \ell -\varepsilon +o(1).
  \end{align*}
But we know that $\lim_{N\to\infty} \lav{N}{b} = \ell$, and it follows that $\lim_{N\to\infty} \lav{N}{\wb}=1$.
Finally, since $\wb=B_\varepsilon+1$, we also get $\lim_{N\to\infty} \lav{N}{B_\varepsilon}=1$.

  \textbf{Step 2} We construct the announced set $B$ as follows. First we fix a decreasing sequence $(\varepsilon_k)$ of positive numbers
  going to 0 as $k\to\infty$. By Step~1, we know that $\lav{N}{B_{\varepsilon_k}}\tend{N}{\infty}1$.
  Then we construct a strictly increasing sequence $(N_k)$ of integers such that
  $\forall k$, $\forall N\ge N_k$, $\lav{N}{B_{\varepsilon_k}} > 1-\varepsilon_k$.
  Finally we define $B$ by $B\cap \{1,\ldots,N_1-1\}\egdef \{1,\ldots,N_1-1\}$, and for each $k\ge1$,
  $B\cap\{N_k,\ldots,N_{k+1}-1\}\egdef B_{\varepsilon_k}\cap\{N_k,\ldots,N_{k+1}-1\}$.
\end{proof}

\subsection{Deriving a density version of the PNT from the logarithmic Chowla conjecture of order~2}\label{s:LogPNT}

\begin{Lemma}\label{l:n2}
Assume that $A\subset\N$ with $\delta(A)=1$. For each $m\in \N$ set
\beq\label{Setsss}
A_m:=\{n\in \N:\:[n/m]\in A\}.
\eeq
Then $\delta(A_m)=1$ for all $m\in \N$.
\end{Lemma}

\begin{proof} Let $m\ge 2$ be fixed.
Note that if $n\in A$, then
\beq\label{S0}
mn+j\in A_m,\quad j=0,1,\dots,m-1,
\eeq
and
\beq\label{InS}
\sum_{j=0}^{m-1}\frac{1}{mn+j}\ge \int_{mn}^{m(n+1)}\frac{dt}{t}=\log(1+1/n)\ge \frac{1}{n}-\frac{1}{2n^2}.
\eeq
Given  $\epsilon\in (0,1)$, we have
$
\frac{1}{\log N}\sum_{n\in A,\;n\le N}\frac{1}{n}\ge 1-\epsilon/2\text{ for }N\ge N_\epsilon$.
Setting $c:=\sum_{k=1}^\infty\frac{1}{k^2}=\frac{\pi^2}{6}$, we may assume that
\[
\frac{c}{2\log {N_\epsilon}}\le \epsilon/2\quad \text{ and }\quad\frac{\log N}{\log {mN}}\ge 1-\epsilon\quad \text{ for all }
 N\ge N_\epsilon.
\]
It follows that for any $N\ge N_\epsilon$, by~(\ref{InS}) and~\eqref{S0}, we have
\begin{align*}
1-\epsilon/2 &\le \frac{1}{\log N}\sum_{n\in A,\;n\le N}\frac{1}{n}\\
             &\le\frac{1}{\log N}\sum_{n\in A,\;n\le N}\left(\sum_{j=0}^{m-1}\frac{1}{mn+j}\right)+
\frac{1}{2\log N}\sum_{n\in A,\;n\le N}\frac{1}{n^2}\\
             &\le\frac{1}{\log N}\sum_{k\in A_m,\;k\le m(N+1)}
\frac{1}{k}+\epsilon/2.
\end{align*}
Therefore, for all $N\geq N_\epsilon$, we have
\[
1\ge \frac{1}{\log {mN}}\sum_{k\in A_m,\;k\le mN}\,\frac{1}{k}=
\frac{\log N}{\log {mN}}\frac{1}{\log {N}}\sum_{k\in A_m,\;k\le mN}
\frac{1}{k}\ge (1-\epsilon)^2.
\]
Letting $\epsilon\to 0$, we obtain
$\lim_{N\to\infty}\, \frac{1}{\log {mN}}\sum_{k\in A_m,\;k\le mN}\,\frac{1}{k}=1$,
and then $\delta(A_m)=1$.
\end{proof}

By Lemmas \ref{l:n2} and
 \ref{Sets} we obtain the following:

\begin{Lemma}\label{l:n2na}
Let $A\subset \N$ with $\delta(A)=1$.
Then there exists $\tilde{A}\subset \N$ with
$\delta(\tilde{A})=1$, such that for any $m\in \N$ there exists $N_m$ satisfying
\beq\label{rrr}
\tilde{A}\cap \{n:\,n\ge N_m\}\subset \bigcap_{k=1}^m A_k
\eeq
(sets $A_k$ are defined by~\eqref{Setsss}).
\end{Lemma}

Given $u:\N\to\C$, let
$U(x):=\sum_{n\le x} u(n)$, for $x\ge 0$,
denote the corresponding summation function.

%

\begin{Lemma}\label{l:mertens}
Let $A\subset \N$ with $\delta(A)=1$, and let  $u:\N\to \C$ such that
\[
\frac{|U(n)|}n\tend{A\ni n}{\infty} 0.
\]
Then there exists $\widetilde{A}\subset A$, $\delta(\widetilde{A})=1$ such that, for each
$a\ge 1$ and
$\vep>0$, we can find $X=X(a,\epsilon)>1$ for which
\[
\forall x\ge X,\ [x]\in \tilde{A} \Longrightarrow \sum_{n\leq a}|U(x/n)|\leq \vep x.
\]
\end{Lemma}

\begin{proof} We set $\widetilde{A}$ as in Lemma~\ref{l:n2na}, so $\delta(\widetilde{A})=1$.
Let $a\ge 1$ be fixed.
Denote $C:=\sum_{n\leq a}\frac{1}{n}$ and choose $K\geq1$ so that
\beq\label{mert1}
A\ni [x]\ge K\;\Rightarrow\;\frac{|U(x)|}{x}
\le \frac{|U([x])|}{[x]}\le \vep/C.
\eeq
Taking $m=[a]$ and using Lemma~\ref{l:n2na}, we choose
 $N_m$ such that \eqref{rrr} holds, i.e.\
$[n], [n/2],\ldots,[n/m]\in A$ whenever $n\in \tilde{A}$ and $n\ge N_m$.
Then for $x\ge (m+1)\max\{N_m,K\}$ with $[x]\in \tilde{A}$, by~\eqref{mert1}, we have
$
\sum_{n\leq a}|U(x/n)|
\leq
\sum_{n\leq a}\frac xn\frac{\vep} C=x\vep.
$
\end{proof}

\begin{Prop}\label{p:toPNT}
The statement {\bf logarithmic} Chowla conjecture (for $\mob$) holds\footnote{Proved by Tao in \cite{Ta1}.} for auto-correlations of length~2 implies that there exists a sequence $A\subset\N$, $\delta(A)=1$, such that
$
\sum_{n\leq x}\Lambda(n)=x+{\rm o}(x)\text{ for }A\ni x\to\infty$.\end{Prop}
\begin{proof}
By Theorem~\ref{banach}~\footnote{Note that if $(c_n)$ is a bounded sequence of complex numbers with $\lim_{N\to\infty}\frac1{\log N}\sum_{n\leq N}\frac1nc_n\ov{c_{n+h}}=0$ for each $h\neq0$ then for
$$R_H((c_n)):=\limsup_{N\to\infty}\frac1{\log N}\sum_{n\leq N}\frac1n\left|\frac1H\sum_{h\leq H}c_{n+h}\right|^2$$
we have $\lim_{H\to\infty}R_H((c_n))=0$.} (or directly by Tao's proof in \cite{Ta}), we obtain that
$
\frac1N\sum_{n\leq N}\mob(n)\to 0$ when $A\ni N\to\infty$,
where $\delta(A)=1$.

Following \cite{Ap}, we repeat the proof that $M(\mob):=\lim_{N\to\infty}\sum_{n\leq N}\mob(n)=0$ imply PNT. We have
$$
\sum_{n\leq x}\Lambda(n)=x-\sum_{q,d; qd\leq x}\mob(d)f(q)+{\rm O}(1)$$
for some arithmetic function $f$. All we need to show is that
$
\sum_{q,d; qd\leq x}\mob(d)f(q)={\rm o}(x)$.
Using summation by parts, one arrives at
\beq\label{summ3}
\sum_{q,d; qd\leq x}\mob(d)f(q)=\sum_{n\leq b}\mob(n)F(x/n)+\sum_{n\leq a}f(n)M(x/n)-F(a)M(b),\eeq
where $ab=x$. Moreover, $F(x)=B\cdot \sqrt x$ for a constant $B>0$. It follows that the first summand in~\eqref{summ3} is bounded by $B_1x/\sqrt a$ for a constant $B_1>B$.
Indeed,
$$
\left|\sum_{n\leq b}\mob(n)F(x/n)\right|\le
B\sqrt{x}\sum_{n\leq b}\frac{1}{\sqrt{n}}\le B_1\sqrt{b}\sqrt{x}=B_1\frac{x}{\sqrt{a}}.
$$
We fix $\vep>0$ and choose $a\geq1$, so that
$
B_1/\sqrt a<\vep$
which yields the first summand $<\vep x$ for all $x\geq1$.
To majorate the second summand, we use
$$
\left|\sum_{n\leq a}f(n)M(x/n)\right|\le
F_a\sum_{n\leq a}|M(x/n)|,\quad F_a:=\max\{|f(n)|:\,n\le a\},
$$
and Lemma~\ref{l:mertens} (with $u=\mob$). Finally, the third summand is majorated in the same way as in \cite{Ap}.\end{proof}

Faculty of Mathematics and Computer Science\\
Nicolaus Copernicus University, Toru\'n, Poland\\
gomilko@mat.umk.pl, mlem@mat.umk.pl\\
Laboratoire de Math\'ematiques Rapha\"el Salem, Normandie Universit\'e\\ CNRS -- Avenue de l'Universit\'e -- 76801
Saint Etienne du Rouvray, France\\
Thierry.de-la-Rue@univ-rouen.fr


\footnotesize
\begin{thebibliography}{99}
\bibitem{Ab-Le-Ru} E.~H.\ El Abdalaoui,  M.\ Lema\'nczyk, T.\ de la Rue, {\em Automorphisms with quasi-discrete spectrum, multiplicative functions and average orthogonality along short intervals}, International Math.\ Res.\ Notices {\bf 14} (2017), 4350-4368.
\bibitem{Ab-Ku-Le-Ru1}
E.~H. El~Abdalaoui, J.~Ku\l{}aga-Przymus,
  M.~Lema{\'n}czyk, and T.~de~la Rue, \emph{The {C}howla and the {S}arnak
  conjectures from ergodic theory point of view}, Discrete Contin.\ Dyn.\ Syst.\
  \textbf{37} (2017),  2899--2944.

\bibitem{Ab-Ku-Le-Ru} E.\ H.\ El Abdalaoui, J.\ Ku\l aga-Przymus,
M. Lema\'nczyk, T. de la Rue,  \emph{M\"obius disjointness
for models of an ergodic system and beyond},
Israel J.\ Math.\ {\bf 228} (2018), 707-751.

\bibitem{Ap}H.\ Apostol, {\em Introduction to Analytic Number Theory},  Undergraduate Texts in Mathematics, Springer, 2010.

\bibitem{Bo-Sa-Zi}
J.~Bourgain, P.~Sarnak, and T.~Ziegler, \emph{Disjointness of {M}\"obius from
  horocycle flows}, From {F}ourier analysis and number theory to {R}adon
  transforms and geometry, Dev. Math., vol.~28, Springer, New York, 2013,
  67--83.
\bibitem{Ce-Si} F.\ Cellarosi, Ya.\ G.\ Sinai, {\em Ergodic properties of square-free numbers}, J.\ European Math.\ Soc.\  {\bf 15} (2013), 1343-1374.

\bibitem{Da-Er} H.\ Davenport, P.\ Erd\"os, {\em On sequences of positive integers}, Acta Arithmetica
{\bf 2} (1936), 147--151.

\bibitem{Dy-Ka-Ku-Le} A. Dymek, S.\ Kasjan, J.\ Ku\l aga-Przymus, M. Lema\'nczyk, {\em $B$-free sets and dynamics},  Trans. Amer. Math. Soc.\ {\bf 370} (2018), 5425-5489.

\bibitem{Fe-Ku-Le}S. Ferenczi, J. Ku\l aga-Przymus, M.\ Lema\'nczyk, {\em  Sarnak's conjecture, what's new}, in:
Ergodic Theory and Dynamical Systems in their Interactions with Arithmetics and Combinatorics,  CIRM Jean-Morlet Chair, Fall 2016,
Editors: S. Ferenczi, J. Ku\l aga-Przymus, M. Lema\'nczyk,
    Lecture Notes in Mathematics 2213,  Springer International Publishing.
\bibitem{Fr-Ho}
N.~Frantzikinakis,  B.~Host, \emph{The logarithmic {S}arnak conjecture for
  ergodic weights}, Annals Math.\ {\bf 187} (2018), 869-931.
\bibitem{Fr-Ho1} N.\ Frantzikinakis, B.\ Host, {\em Furstenberg systems of bounded multiplicative functions and applications},
    to appear in International Math.\ Res.\ Notices, arXiv:1804.08556.
\bibitem{Go-Kw-Le} A.\ Gomilko, D.\ Kwietniak, M.\ Lema\'nczyk,  {\em Sarnak's conjecture implies the Chowla conjecture along a subsequence}, in:
Ergodic Theory and Dynamical Systems in their Interactions with Arithmetics and Combinatorics,  CIRM Jean-Morlet Chair, Fall 2016,
Editors: S. Ferenczi, J. Ku\l aga-Przymus, M. Lema\'nczyk,
    Lecture Notes in Mathematics 2213,  Springer International Publishing, pp. 418, arXiv:1710.07049.

\bibitem{Ka-Le-Ul}A.\ Kanigowski, M.\ Lema\'nczyk, C.\ Ulcigrai,
{\em On disjointness of some parabolic flows}, arXiv:1810.11576.


\bibitem{Sa}
P.~Sarnak, {\em Three lectures on the {M}{\"o}bius function, randomness and
  dynamics}, \url{http://publications.ias.edu/sarnak/}.


\bibitem{Ta00} T.\ Tao,
https://terrytao.wordpress.com/2012/10/14/the-chowla-conjecture-and-the-sarnak-conjecture/


\bibitem{Ta1}
T.~Tao, \emph{The logarithmically averaged {C}howla and {E}lliot conjectures
  for two-point correlations}, Forum Math. Pi \textbf{4} (2016).
\bibitem{Ta2}T.\ Tao, \emph{Equivalence of the logarithmically averaged {C}howla and {S}arnak conjectures}, Number Theory -- {D}iophantine Problems, Uniform
  Distribution and Applications: Festschrift in Honour of Robert F. Tichy's
  60th Birthday (C.~Elsholtz and P.~Grabner, eds.), Springer International
  Publishing, Cham, 2017, 391--421.
\bibitem{Ta} T. Tao, \url{https://terrytao.wordpress.com/2017/10/20/the-logarithmically-averaged-and-non-logarithmically-averaged-chowla-conjectures/}
\bibitem{Ta-Te} T. Tao, J. Teravainen, {\em The structure of correlations of multiplicative functions at almost all scales, with applications to the Chowla and Elliot conjectures},
 arXiv 1809.02518.
\bibitem{Ta-Te1} T. Tao, J.\ Teravainen, {\em Odd order cases of the logarithmically averaged Chowla conjecture}, arXiv:1710.02112.
\end{thebibliography}
\normalsize
\end{document}